\documentclass[12pt]{article}

\usepackage{amsthm}
\usepackage{amsmath}
\usepackage{amsfonts}
\usepackage{amssymb}
\usepackage{amscd}
\usepackage{mathrsfs} % For nice script letters
\usepackage[all]{xypic} 
\newtheorem{thm}{Theorem}[section]

\newtheorem{prop}[thm]{Proposition}

\theoremstyle{definition}
\newtheorem{defn}[thm]{Definition}%[section]

\theoremstyle{definition}
\newtheorem{example}[thm]{Example}

\begin{document}

\bibliographystyle{plain}
\title{Relations among conditional probabilities}
\author{Jason Morton}

\maketitle
\begin{abstract}
We describe a Gr\"obner basis of relations among conditional probabilities in a discrete probability space, with any set of conditioned-upon events. They may be specialized to the partially-observed random variable case, the purely conditional case, and other special cases. We also investigate the connection to generalized permutohedra and describe a ``conditional probability simplex.''
\end{abstract}

%--------------------------------------------------------------------
\newcommand{\scrE}{\mathscr E}
\newcommand{\C}{\mathbb{C}}
\newcommand{\Z}{\mathbb{Z}}
\newcommand{\R}{\mathbb{R}}
\newcommand{\N}{\mathbb{N}}
\newcommand{\PP}{\mathbb{P}}
\newcommand{\TT}{\mathbb{T}} 
\newcommand{\mcA}{\mathcal{A}}
\newcommand{\mcB}{\mathcal{B}}
\newcommand{\mcZ}{\mathcal{Z}}
\newcommand{\mbff}{\mathbf{f}}
\newcommand{\bfp}{\mathbf{p}}
\newcommand{\bfq}{\mathbf{q}}

\newcommand{\rowspan}{\operatorname{rowspan}}
\newcommand{\rank}{\operatorname{rank}}
\newcommand{\frkF}{\mathfrak{F}}
\newcommand{\conv}{\operatorname{conv}}
\newcommand{\mconv}{\operatorname{mconv}}
\newcommand{\functo}{\rightarrow}
\newcommand{\ra}{\rightarrow}
\newcommand{\supp}{\operatorname{supp}}
\newcommand{\initial}{\operatorname{in}}
\newcommand{\union}{\cup}
\newcommand{\intersection}{\cap}
\newcommand{\pos}{\operatorname{pos}}
\newcommand{\nbhd}{\operatorname{nbhd}}
\newcommand{\Spec}{\operatorname{Spec}}
\newcommand{\isomorphic}{\cong}
\newcommand{\isom}{\isomorphic}
\newcommand{\piI}{p_{i|I}}
\newcommand{\piJ}{p_{i|J}}
\newcommand{\piK}{p_{i|K}}
\newcommand{\pjI}{p_{j|I}}
\newcommand{\pjJ}{p_{j|J}}
\newcommand{\pjK}{p_{j|K}}
\newcommand{\pkK}{p_{k|K}}
\newcommand{\pII}{p_{I|I}}
\newcommand{\pJJ}{p_{J|J}}
\newcommand{\pJK}{p_{J|K}}
\newcommand{\sinpioverthree}{0.866025404}

\section{Relations among conditional probabilities}

 In 1974, Julian Besag \cite{Besag1974} %"Spatial Interaction and the Statistical Analysis of Lattice Systems".
 discussed the ``unobvious and highly restrictive consistency conditions'' among conditional probabilities.  %The problem of discovering these conditions is interesting both in general (i.e. for any conditional probability distribution) and after fixing a particular statistical model.  
In this paper we give an answer in the discrete case to the question  {\em  
What conditions must a set of conditional probabilities satisfy in order to be compatible with some joint distribution?
}

Let  $\Omega = \{1, \dots, m\}$ be a finite set of singleton events, and let $p=(p_1, \dots, p_m)$ be a probability distribution on them.  
Let $\scrE$ be a set of observable events which will be conditioned on, each a set of at least 2 singleton events.  Then for events $I \subset J$, $J$ in $\scrE$,  we can assign conditional probabilities for the chance of $I$ given $J$, denoted $p_{I|J}$.  Settling Besag's question then becomes a matter of determining the {\em relations} that must hold among the quantities $p_{I|J}$. 
For example, Besag gives the relation (see also \cite{Besag1972a}), %,derived by the Brooks expansion,
\begin{equation} \label{eq:BesagRelation}
\frac{P(\mathbf{x})}{P(\mathbf{y})} = \prod_{i=1}^n \frac{P(x_i|x_1, \dots, x_{i-1}, y_{i+1}, \dots, y_n)}{P(y_i|x_1, \dots, x_{i-1}, y_{i+1}, \dots, y_n)}.
\end{equation}
Since there are in general infinitely many such relations, we would like to organize them into an ideal and provide a nice basis for that ideal.  A quick review of language of ideals, varieties, and Gr\"obner bases appears in Geiger et al. \cite[p. 1471]{TAGM} and more detail in Cox et al. \cite{Cox1997}.   In Theorem \ref{thm:CPuniversalGB}, we generalize relations such as  \eqref{eq:BesagRelation} and Bayes' rule to give a universal Gr\"obner basis of this ideal, a type of basis with useful algorithmic properties.

The second result generalized in this paper is due to Mat\'{u}\v{s} \cite{Matus2003Conditional}.  This states that the space of conditional probability distributions $(p_{i|ij})$ conditioned on events of size two maps homeomorphically onto the permutohedron.  In Theorem \ref{thm:MultiMomentMap}, we generalize this result to arbitrary sets $\scrE$ of conditioned-upon events.  The resulting image is a {\em generalized permutohedron} \cite{Postnikov, Ziegler1995}.  This is a polytope which provides a canonical, conditional-probability analog to the probability simplex under the correspondance provided by toric geometry \cite{GBCP} and the theory of exponential families.

Work on the subject of relations among conditional probabilities has primarily focused on the case where the events in $\scrE$ correspond to observing the states of a subset of $n$ random variables.  
%The Hammersley-Clifford theorem provides a partial answer to this question in the case of undirected graphical models with strictly positive probability distributions \cite{Besag1974}.  These are not the same graphs as those depicted in this paper. 
Arnold et. al. \cite{Arnold1999} develop the theory for both discrete and continuous random variables, particularly in the case of two random variables, and cast the compatibility of two families of conditional distributions as a solutions to a system of linear equations.  Slavkovic and Sullivant \cite{SlavkovicSullivant2006} consider the case of compatible full conditionals, and compute related unimodular ideals. %A little more detail on these will be given after we have developed the necessary terminology in Section \ref{sec:caseofrv}. 
%Mat\'u{\v s} \cite{Matus2003Conditional} considers the case of $m$ discrete events where $\scrE = \{ I \subset [m]: |I| \geq 2\}$ (all possible compound events).  He gives a geometric characterization of the space of conditional probabilities through projection to the case $\scrE = \{ I \subset [m]: |I|=2\}$.

This paper is organized as follows.
In Section \ref{sec:CPDist}, we introduce some necessary definitions.  
In Section \ref{sec:UnivGB}, we give compatibility conditions in the general case of $m$ events in a discrete probability space, with any set $\scrE$ of conditioned-upon events.  These conditions come in the form of a universal Gr\"obner basis, which makes them particularly useful for computations: as a result, they may be specialized to the partially observed random variable case, the purely conditional case, and other special cases simply by changing $\scrE$.  
In \cite{3CEX, CRT}, we have seen that permutohedra and generalized permutohedra \cite{Postnikov} play a central role in the geometry of conditional independence; the same is true of conditional probability.
The geometric results of Mat\'u{\v s} \cite{Matus2003Conditional} map the space of conditional probability distributions (Definition \ref{def:CPdist}) for all possible conditioned events $\scrE = \{ I \subset [m]: |I| \geq 2\}$ onto the permutohedron $\mathbf{P}_{m-1}$.  See Figure \ref{fig:P4updown} for a diagram of the $3$-dimensional permutohedron.  In Section \ref{sec:CPmoment}, we will discuss how to extend this result to general $\scrE$, in which case we obtain generalized permutohedra as the image. 
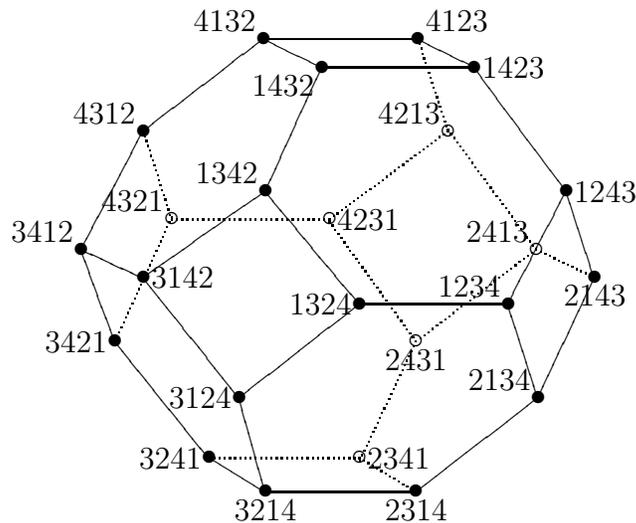
\begin{figure}[htb]\label{UpDown4}
\[
 \begin{xy}<25mm,0cm>:
%Permutohedron with n=4
%Points and labels
(1,0)  ="3214"  *+!U{3214} *{\bullet};
(1.8,0) ="2314"  *+!U{2314} *{\bullet};
(.7,.18)  ="3241"  *+!R{3241} *{\bullet};
(1.5,.18)  ="2341"  *+!L{2341} *{\circ}; %in back
(.86,.5)  ="3124"  *+!R{3124} *{\bullet};
(2.45,.5)  ="2134"  *+!DR{2134} *{\bullet};
(.2,.8)  ="3421"  *+!R{3421} *{\bullet};
(1.8 ,.8)  ="2431"  *+!U{2431} *{\circ}; %in back
(1.5,1)  ="1324"  *+!R{1324} *{\bullet};
(2.29,1)  ="1234"  *+!DR{1234} *{\bullet};
(.35,1.14)  ="3142"  *+!L{3142} *{\bullet};
(2.75,1.14)  ="2143"  *+!U{2143} *{\bullet};
(.02,1.29)  ="3412"  *+!DR{3412} *{\bullet};
(2.44,1.29)  ="2413"  *+!DR{2413} *{\circ}; %in back
(.5,1.45)  ="4321"  *+!DR{4321} *{\circ}; %in back
(1.34,1.45)  ="4231"  *+!L{4231} *{\circ}; %in back
(1,1.6)  ="1342"  *+!DR{1342} *{\bullet};
(2.6,1.6)  ="1243"  *+!L{1243} *{\bullet};
(.35,1.92)  ="4312"  *+!DR{4312} *{\bullet};
(1.97,1.92)  ="4213"  *+!DR{4213} *{\circ}; %in back
(1.3,2.26)  ="1432"  *+!UR{1432} *{\bullet};
(2.11,2.26)  ="1423"  *+!L{1423} *{\bullet};
(.99,2.41)  ="4132"  *+!DR{4132} *{\bullet};
(1.81,2.41)  ="4123"  *+!DL{4123} *{\bullet};
%EDGES %Squares
%Bottom square
"3214";"2314" **@{-}; % 2 indep 3
"3241";"2341" **@{.}; % 2 indep 3
"3241";"3214" **@{-}; % 1 indep 4 | 23
"2341";"2314" **@{.}; % 1 indep 4 | 23
%Right square
"2134";"1234" **@{-}; % 1 indep 2 |
"2143";"1243" **@{-}; % 1 indep 2 |
"1234";"1243" **@{-}; % 3 indep 4 | 12
"2134";"2143" **@{-}; % 3 indep 4 | 12
%Top square
"4132";"4123" **@{-}; % 2 indep 3 | 14
"1432";"1423" **@{-}; % 2 indep 3 | 14
"1432";"4132" **@{-}; % 1 indep 4 |
"4123";"1423" **@{-}; % 1 indep 4 |
%Left square
"4312";"3412" **@{-}; % 3 indep 4 |
"4321";"3421" **@{.}; % 3 indep 4 |
"4312";"4321" **@{.}; % 1 indep 2 | 34
"3412";"3421" **@{-}; % 1 indep 2 | 34
%Back square
"4213";"2413" **@{.}; % 2 indep 4 |
"4231";"2431" **@{.}; % 2 indep 4 |
"4213";"4231" **@{.}; % 1 indep 3 | 24
"2413";"2431" **@{.}; % 1 indep 3 | 24
%Front square
"1342";"1324" **@{-}; % 2 indep 4 | 13
"3142";"3124" **@{-}; % 2 indep 4 | 13
"1342";"3142" **@{-}; % 1 indep 3 |
"1324";"3124" **@{-}; % 1 indep 3 |
%Rest of edges
"2314";"2134" **@{-}; % 1 indep 3 | 2
"3124";"3214" **@{-}; % 1 indep 2 | 3
"3421";"3241" **@{-}; % 2 indep 4 | 3
"3412";"3142" **@{-}; % 1 indep 4 | 3
"1324";"1234" **@{-}; % 2 indep 3 | 1
"1432";"1342" **@{-}; % 3 indep 4 | 1
"4312";"4132" **@{-}; % 1 indep 3 | 4
"1423";"1243" **@{-}; % 2 indep 4 | 1
"2341";"2431" **@{.}; % 3 indep 4 | 2
"4321";"4231" **@{.}; % 2 indep 3 | 4
"2413";"2143" **@{.}; % 1 indep 4 | 2
"4123";"4213" **@{.}; % 1 indep 2 | 4
%"";"" **@{.}; %  indep  |
\end{xy}
\]
\caption{The permutohedron ${\bf P}_4$.} \label{fig:P4updown}
\end{figure}
This will be accomplished using a version of the moment map of toric geometry (Theorem \ref{thm:MomentMap}).
In Section \ref{sec:caseofrv}, we discuss how to specialize our results to the case of $n$ partially observed random variables, including as an example how to recover the relation (\ref{eq:BesagRelation}).
Finally, in Section \ref{sec:BayesExplanation} we use this specialization to explain the relationship of Bayes' rule to our constructions.  In the Appendix we recall a few necessary facts about toric varieties.

\section{Conditional probability distributions} \label{sec:CPDist}
Let $\scrE$ be a collection of subsets $I$, with $|I| \geq 2$, of $[m]=\Omega = \{1, \dots, m\}$. % containing $[m]$.  
Let $\C[\scrE]$ denote the {\em event algebra},  the polynomial ring with indeterminates $p_{i|I}$ for all $I \in \scrE$ and $i \in I$, i.e. one unknown for each elementary conditional probability.  Then we denote by
\[
\|\scrE\| = \sum_{I \in \scrE} |I|
\] 
the number of variables of $\C[\scrE]$.  We write $p_i$ for $p_{i|[m]}$ when $[m] \in \scrE$.
The unknowns of $\C[\scrE]$ are meant to represent conditional probabilities, as we now explain.  The set $\{1, \dots, m\}$ indexes the $m$ disjoint events, and a point $(p_1, \dots, p_m) \in \R^m_{\geq 0}$ with $\sum_j p_j = 1$ represents a probability distribution on these events. % $([n], 2^{[n]},P)$.
When $p_j >0$ for all $j$, the {\em conditional probability} of event $i$ given event $I$ containing it is 
\begin{equation} \label{eq:simpleCP}
p_{i|I} = \frac{p_i}{\sum_{j \in I} p_j}.
\end{equation}
To extend this notion to the case $P(I) = \sum_{j \in I} p_j = 0$, and to be able to deal with multiple conditioning sets, we make the following standard definition \cite{Billingsley}, considered in this form by Mat\'u{\v s} \cite{Matus2003Conditional}.

\begin{defn} \label{def:CPdist}
A {\em conditional probability distribution} for $\scrE$ is a point $(p_{i|I} \; : \; i \in I \in \scrE) \in \R^{\|\scrE\|}_{\geq 0}$ such that for all $J, K \in \scrE$ with $J \subset K$,
\begin{itemize}
\item[(i)] $\sum_{i\in J} p_{i|J} = 1$
\item[(ii)] for all $i \in J$, $p_{i|K} = p_{i|J} \sum_{j \in J} p_{j|K}$.
\end{itemize}
\end{defn}

Observe that (ii) is a relative version of \eqref{eq:simpleCP}, as  \eqref{eq:simpleCP} follows from (ii) with $K=[m]$,  $J=I$, and $\sum_{i \in I} p_i \neq 0$.  If on the other hand $\sum_{j \in J} p_{j|K} = 0$, the whole probability simplex $\Delta_J : = \{ (\pjJ)_{j \in J} : \pjJ \geq 0, \sum_{j\in J} \pjJ =1 \}$ satisfies the definition.  This freedom is known in probability theory as {\em versions of conditional probability} \cite{Billingsley}.
In algebraic geometry, this corresponds to the notion of a blow-up, \cite{Harris1995} and the simplex $\Delta_J$ to the exceptional divisor.
Before we give a homogenized version of Definition \ref{def:CPdist}, we consider the homogenized version of probability.
\subsection{A projective view of probability}\label{sec:ProjProb}
Consider a probability space with $m$ disjoint atomic events $([m], 2^{[m]},P)$.  The space of probability distributions $P$ on them is typically represented as a {\em probability simplex}, where each $P(i)$ is a coordinate $p_i$ such that $p_i\geq 0$ and $\sum_i p_i =1$.  We will be describing families of probability distributions in terms of {\em algebraic varieties}, and we prefer to think of points $(p_1: \cdots : p_m)$ as lying in complex projective space.  This is equivalent to letting $V=\C\{e_1, \dots,e_m\} \isom \C^m$ be the complex vector space spanned by the outcomes (singleton events) and considering points $p \in \PP V$ as representing mixtures over outcomes or probability distrubutions.   There are two ways to match up the notion of the probability simplex with that of complex projective space.  One way to do so, {\em restriction}, identifies the probability simplex $\Delta_{m-1}$ with the real, positive part of the affine open $\sum_i y_i \neq 0$ of the $\PP^{m-1}$ with homogeneous coordinates $(y_1:y_2: \cdots :y_m)$ as illustrated in Figure \ref{fig:projprobsimplex}.
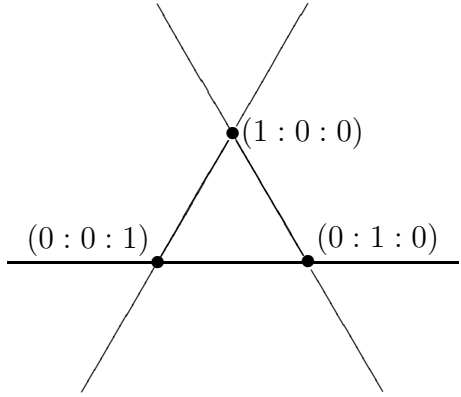
\begin{figure}[htb]
\[ 
   \begin{xy}<2cm,0cm>:
(2,0) *+!DR{(0:0:1)} *{\bullet};
      p+(.5,\sinpioverthree) *+!L{(1:0:0)} *{\bullet} **@{-};
        p+(.5,\sinpioverthree) **@{-},
        p+(-1,-1.72)  **@{-},
      p+(.5,-.86) *+!DL{(0:1:0)} *{\bullet} **@{-};
        p+(.5,-.86) **@{-},
        p+(-1,1.72) **@{-},
      p+(-1,0) **@{-};
        p+(-1,0) **@{-},
        p+(2,0) **@{-},
   \end{xy}
\]
\caption{Probability simplex in the projective plane} \label{fig:projprobsimplex}
\end{figure}

%This requires that we through away a lot of projective space.  
Alternatively we can use {\em projection}, equivalent in the special case that $(y_1: \cdots : y_m) \in \Delta_{m-1}$, via the moment map (Theorem \ref{thm:MomentMap}).
The identity matrix $\mcA = I_m$ comprised of standard unit vectors $e_i$ defines the probability simplex $\Delta_{m-1}=\conv(\mcA)$. The toric variety $Y_{\mcA}$ is then the projective space $\PP^{m-1}$ and the moment map is:
\[
\mu: \PP^{m-1} \functo \Delta_{m-1}
\]
\[
\mu((y_1: \dots : y_m))= \frac{1}{\sum_i |y_i|} |y_i| e_i
\]
%Where the $e_i$ are the standard unit vectors in this case, but more generally are the lattice points of $\PT$ or just its extreme vertices, depending on the application.  
The moment map $\mu$ is the identity map on the probability simplex, but allows us to define a point on the probability simplex for more general points in complex projective space.  The fiber over any of these points is the torus $(\mathbb{S}^1)^n$, a product of $m$ unit circles, since $\mu(y_1: \dots, :y_m) = \mu(e^{i\theta_1} y_1 : \dots : e^{i \theta_m} y_m )$.  A similar point of view appears in quantum physics; here $V=\C\{x: x \; \text{a classical state}\}$ is the Hilbert space representing quantum state and the modified moment map $\mu' (y): \frac{1}{\sum_i |y_i|^2} |y_i|^2 e_i$ defines the probability of observing a classical state (singleton event) \cite{Nielsen2000}.

One interpretation of this freedom is that it suggests there are circumstances where allowing probabilities to be negative and even complex in intermediate computations might be useful.  This may seem odd, but it can be argued that negative probabilities are %not only useful, but 
already implicitly employed \cite{Feynman1987}.  For example, characteristic function methods implicitly write a density as a linear combination of basis functions with ranges unrestricted to $\R_{\geq 0}$.
%On the other hand, many algorithms work better when constrained at each step to output probabilities, rather than simply projecting to the probability simplex at the end.
Even if we are uncomfortable with such interpretations, the compactification and homogenization can simply be viewed as a convienient algebraic trick to make it easy to determine the relations among conditional probabilities we are ultimately interested in.  Moreover, for most purposes $\C$ can be replaced with $\R$ \cite{TAGM} as the base field for our ring, and these relations are unchanged.

\subsection{Homogeneous conditional probability}
 Analogously to the projective version of probability in Section \ref{sec:ProjProb}, where we replaced the requirement that probabilities $p_1, \dots, p_m$ sum to one with viewing them as coordinates of a point in projective space, we now define a multihomogeneous version of Definition \ref{def:CPdist}.  
Now, a conditional probability distribution is represented by a point in the product of projective spaces.  This product has one $\PP^{|I|-1}$ for each event $I \in \scrE$ which is conditioned upon, and each factor space $\PP^{|I|-1}$ is equipped with homogeneous coordinates $(p_{i_1 | I} : \cdots :  p_{i_{|I|} | I})$.
%Defined over $\C$ instead of $\R_{\geq 0}$, it gets its name from \cite{Renyi}.

\begin{defn} \label{def:Renyiprob}
A {\em projective conditional probability distribution} for $\scrE$ is a point $\mathbf{p}=((p_{i_1 | I}: \cdots : p_{i_{|I|} |I}), I \in \scrE)$ inside $\prod_{I \in \scrE} \PP^{|I|-1}$ such that for all $J, K \in \scrE$ and $i \in J \subset K$, 
\[
(\sum_{j \in J} \pjJ) \piK = \piJ (\sum_{j \in J} \pjK)
\]
\end{defn}
\noindent
Definition \ref{def:Renyiprob} specifies the following ideal in the event algebra $\C[\scrE]$:
\[
J_\scrE = \langle (\sum_{j \in J} \pjJ) \piK - \piJ (\sum_{j \in J} \pjK) \; : \; J, K \in \scrE, \;  i \in J \subset K \rangle.
\]
This ideal consists of all polynomial relations that a point $P=(\piI)$ in $\prod_{I \in \scrE} \PP^{|I|-1}$ must satisfy to be a projective conditional probability distribution.  In particular, any honest conditional probability distribution must satisfy these.  If we denote by $\{e_I : I \in \scrE\}$ a basis of $\Z^{|\scrE|}$, this ideal $J_\scrE$ is multihomogeneous with respect to the grading $\deg(\piI) = e_I$ (see e.g. \cite{MillerSturmfels} for more on such gradings).
In what follows, it will be convenient to abbreviate
$
\pJJ:=\sum_{j \in J} \pjJ.
$ 
Thus $\pJJ$ would be equal to $1$ for honest distributions, by Definition \ref{def:CPdist}, but here we regard it as a linear form in $\C[\scrE]$.  
Let $\alpha_{\scrE}$ denote the product $\prod_{i \in I \in \scrE} \piI$ of all of the $\|\scrE\|$ variables in $\C[\scrE]$, and let $\beta_{\scrE}$ denote the product $\prod_{I \in \scrE} \pII$.  The {\em saturation} $(I:f^\infty)$ of an ideal $I$ is the ideal generated by all polynomials $g$ such that $f^m g\in I$ for some $m$ \cite{GBCP}. Now we define the ideal $I_\scrE$, when $[m] \in \scrE$, by the saturation 
\[
I_\scrE := (J_\scrE : (\alpha_{\scrE} \beta_{ \scrE} )^{\infty}).
\]
When $[m] \notin \scrE$, let $\scrE' = \scrE \union [m]$ and set  $I_\scrE := I_{\scrE'} \intersection \C[\scrE]$.  The purpose of saturation is to make sure the desired behavior occurs when some coordinates are zero; for example, it is necessary to move between the conditional independence ideals \cite{TAGM} generated by expressions $P(X\!=\!x,Y\!=\!y | Z\!=\!z) - P(X\!=\!x | Z\!=\!z) P(Y\!=\!y | Z\!=\!z)$ and by the cross product differences $P(x,y,z)P(x',y',z)-P(x,y',z)P(x',y,z)$ algebraically without assuming anything about the positivity of the probabilities in question.

In the next section, we describe a  matrix  $\mcA_{G}$ such that $I_\scrE$ arises as the toric ideal  $I_{\mcA_G}$ (Section \ref{sec:toricvar}).  Our first main result will be a universal Gr\"obner basis for the toric ideal $I_\scrE$.   Gr\"obner bases, particularly universal Gr\"obner bases, have many algorithmic properties that make them a very complete description of an ideal.  Cox, Little, and O'Shea \cite{Cox1997} give an accessible overview; see also \cite{GBCP,Greuel2002}.

\section{A universal Gr\"obner basis for relations among conditional probabilities} \label{sec:UnivGB}

A {\em Bayes binomial} in $\C[\scrE]$ is a binomial relation of the form
\[
\piK \pjJ - \pjK \piJ
\]
for $i,j \in J \subseteq K$, with $J,K \in \scrE$. Let $I_{{\rm Bayes}(\scrE)}$ denote the ideal they generate. Bayes binomials get their name because they come from Bayes' rule; 
more explanation 
is given in Section \ref{sec:BayesExplanation}.

\begin{prop} \label{prop:BayesSat}
The ideal generated by the Bayes binomials contains $J_\scrE$ and is contained in the saturation of $J_\scrE$ by the probabilities that would sum to one (where again $\beta_\scrE = \prod_{I \in \scrE} p_{I|I}$):
\[
J_\scrE \subseteq I_{{\rm Bayes}(\scrE)} \subseteq (J_\scrE : (\beta_{ \scrE})^{\infty})
\]
 and in particular, $I_{{\rm Bayes}(\scrE)} \subseteq I_\scrE $.
\end{prop}
\begin{proof}
The ideal $J_\scrE$ is generated by the degree-2 polynomials $\pJJ \piK - \piJ \pJK$ for $J,K \in \scrE$ and $i \in J \subseteq K$. For each $i,j \in J$, we have $a=\pjJ (\pJJ \piK - \piJ \pJK)$ and $b=\piJ(\pJJ \pjK - \pjJ \pJK)$ in  $J_\scrE$, so $a-b = \pJJ ( \pjJ \piK - \pjK \piJ)$ is in $J_\scrE$ and $I_{{\rm Bayes}(\scrE)} \subseteq (J_\scrE : (\beta_{\scrE} )^{\infty})$.  For the first inclusion, if $\pJJ \piK - \piJ \pJK$ is a generator of $J_\scrE$, we may write it as an element $\sum_{j \in J} (\piK \pjJ - \pjK \piJ)$ of $I_{{\rm Bayes}(\scrE)}$.
\end{proof}

Our universal Gr\"obner basis of $I_\scrE$ will be given combinatorially by the cycles of a labeled bipartite graph $G(\scrE)$, defined as follows:\\
{\bf Vertices:} one vertex $u_I$ for each $I \in \scrE$ and one vertex $v_i$ for each $i \in \union_{I \in \scrE} I$\\
{\bf Edges:} a directed edge $u_I \rightarrow v_i$ for each $I \in \scrE$ and $i \in I$\\
{\bf Edge Labels:}  the edge $u_I \rightarrow v_i$ is labeled with the indeterminate $p_{i|I}$.\\

For example, with $n=4$, the labeled graph $G$ for $\scrE = \{ \{1,2\}, \{1,2,3\},$ $\{1,2,3,4\} \}$ is shown in Figure \ref{fig:bipartite12-123-1234}.
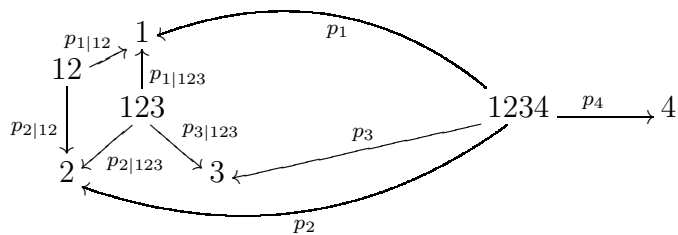
\begin{figure}[htb]
\[ \SelectTips{cm}{10}
   \begin{xy}<1cm,0cm>:
(0,0) *+!{123}="123";
      p+(0,1) *+!{1}="1",
      p+(1,-.86) *+!{3}="3",
      p+(-1,-.86) *+!{2}="2";
{\ar@{->}_{p_{1|123}} "123"; "1"};
{\ar@{->}^{p_{2|123}} "123"; "2"};
{\ar@{->}^{p_{3|123}} "123"; "3"};
(-1,.5) *+!{12}="12"; 
{\ar@{->}^{p_{1|12}} "12"; "1"};
{\ar@{->}_{p_{2|12}} "12"; "2"};
(5,0) *+!{1234}="1234";
      p+(2,0) *+!{4}="4";
{\ar@{->}^{p_{4}} "1234"; "4"};
%"1234";"1" **\crv{(0,5)} *@{>};
{\ar@{->}@/_{2pc}/^{p_{1}} "1234"; "1"};
{\ar@{->}_{p_3} "1234"; "3"};
{\ar@{->}@/^{2pc}/^{p_2} "1234"; "2"};
   \end{xy}
\]
\caption{Bipartite graph for $\scrE = \{ \{1,2\}, \{1,2,3\}, \{1,2,3,4\} \}$.} \label{fig:bipartite12-123-1234}
\end{figure}
Each oriented cycle  $C$ in the undirected version of $G$ defines a binomial $f_C$ as follows: each edge label is on the positive side of the binomial if its edge is directed with the cycle, and on the negative if against.  For example, in the graph in Figure \ref{fig:bipartite12-123-1234}, consider the cycle $(1234,3,123,1,1234)$.  The edges $p_3$ and $p_{1|123}$ are directed with the cycle and the edges $p_{3|123}$ and $p_1$ are directed against, so the corresponding binomial is $p_3 p_{1|123} - p_{3|123} p_1$.  For a higher degree example, with $n=3$ and $\scrE = \{ \{1,2\},\{1,3\},\{2,3\},\{1,2,3\}\}$, we get $p_{1|12}p_{3|13}p_{2|23}-p_{2|12}p_{3|23}p_{1|13}$ from the outer cycle, as shown in Figure \ref{fig:bipartite12-13-23-123}.
\begin{figure}[htb]
\[ \SelectTips{cm}{10}
   \begin{xy}<1cm,0cm>:
(0,0);
      p+(0,1) *+!{1}="1",
      p+(1,-.7) *+!{3}="3",
      p+(-1,-.7) *+!{2}="2";
(-1,.5) *+!{12}="12"; 
(1,.5) *+!{13}="13"; 
(0,-1.3) *+!{23}="23"; 
{\ar@{->}^{p_{1|12}} "12"; "1"};
{\ar@{->}_{p_{2|12}} "12"; "2"};
{\ar@{->}^{p_{3|13}} "13"; "3"};
{\ar@{->}_{p_{1|13}} "13"; "1"};
{\ar@{->}^{p_{2|23}} "23"; "2"};
{\ar@{->}_{p_{3|23}} "23"; "3"};
   \end{xy}
\]
\caption{Outer cycle of the bipartite graph for $\scrE = \{ \{1,2\},\{1,3\},\{2,3\},\{1,2,3\}\}$.} \label{fig:bipartite12-13-23-123}
\end{figure}
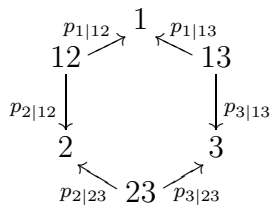
A cycle is {\em induced} if it has no chord.
%Number of cycles when $\scrE$ is all possible sets?

\begin{thm} \label{thm:CPuniversalGB}  
The binomials defined by the cycles of $G(\scrE)$ give a universal Gr\"obner basis for $I_\scrE$.  Moreover, $I_\scrE$ is generated by the induced cycle binomials, though not necessarily as a Gr\"obner basis.
\end{thm}

In order to prove Theorem \ref{thm:CPuniversalGB}, we first need to recall some facts about unimodular toric ideals, of which $I_\scrE$ is an example.  
Unimodular matrices and unimodular toric ideals are defined and characterized as follows, following Sturmfels \cite{GBCP}.  
A {\em triangulation} of $\mcA$ is a collection $\frkF$ of subsets $\mcB$ of the columns of $\mcA$ such that $\{ \pos(\mcB) : \mcB \in \frkF \}$ is the set of cones in a simplicial fan with support $\pos(\mcA)$.  
A triangulation of $\mcA$ is {\em unimodular} if the normalized volume \cite{GBCP} is equal to one for all maximal simplices $\mcB$ in the triangulation%(or equivalently, ${\rm in}_{\prec} I_\mcA$ is squarefree for the triangulation $\mcA_{\prec}$ \cite{GBCP}.  
.  The matrix $\mcA$ is a {\em unimodular} matrix if all triangulations of $\mcA$ are unimodular.  We define a unimodular toric ideal in the following definition-proposition.

\begin{prop} \cite{GBCP} \label{prop:unimodularTFAE}
A toric ideal $I_\mcA$ is called {\em unimodular} \index{toric ideal!unimodular} if any of the following equivalent conditions hold.
\begin{itemize}
\item[(i)] Every reduced Gr\"obner basis of $I_\mcA$ consists of squarefree binomials,
\item[(ii)] $\mcA$ is a unimodular matrix, 
\item[(iii)] all the initial ideals of $I_\mcA$ are squarefree.
\end{itemize}
\end{prop}

A special class of unimodular matrices are those coming from bipartite graphs \cite{AokiTakemura2002, SlavkovicSullivant2006}. 
Let $G=(U,V,E)$ be a bipartite graph. 
In our case, $G(\scrE)$ has 
\begin{equation} \label{eq:vertexsets}
U = \{ u_I : I \in \scrE \} \text{ and } V = \{ v_i : i \in \union_{I \in \scrE} I \}.
\end{equation}
Let $\mcA$ be the vertex-edge incidence matrix of $G$:  The rows of $\mcA$ are labeled $u_1, \dots, u_{|U|},$ $v_1, \dots, v_{|V|}$, the columns are labeled with the edges, and $a_{ij}$ is $1$ if vertex $i$ is in edge $j$ and zero otherwise.
For a cycle $C$ in the graph, the cycle binomial $f_C$ is defined (up to sign) as above.
Let $\pi_{\mcA}$ be the map $\R^{\|\scrE\|} \ra \R^{|U|+|V|}$ defined by applying $\mcA$.   
We say $u \in \ker (\pi_{\mcA})$ is a {\em circuit} if $\supp(u)$ is minimal with respect to inclusion in $\ker(\pi_\mcA)$ and the coordinates of $u$ are relatively prime \cite{GBCP}.  Equivalently, a circuit is an irreducible binomial $x^{u+} - x^{u-}$ of the toric ideal $I_\mcA$ with minimal support. The {\em Graver basis} of the ideal $I_{\mcA}$ consists of all circuits.
For $\mcA$ from a bipartite graph, the circuits of $\mcA$ are precisely the cycle binomials of  the graph \cite{Schrijver1998,SlavkovicSullivant2006}.  Additionally, a Graver basis is also a universal Gr\"obner basis in the case of unimodular toric varieties (Proposition 8.11 of \cite{GBCP}).  We summarize these results in the following proposition.  %Check again: circuit=cycle not circuit = induced cycle

\begin{prop} \label{prop:bipartiteUnimodular}
The vertex-edge incidence matrix $\mcA$ of a bipartite graph $G=(U,V,E)$ is unimodular, so $I_\mcA$ is a unimodular toric ideal.  The cycle binomials of $G$ are the circuits of $\mcA$, and therefore define the Graver basis of $I_{\mcA}$.  In particular, they give a universal Gr\"obner basis for $I_\mcA$.
\end{prop}

Now we are able to prove our theorem.

\begin{proof}[Proof of Theorem \ref{thm:CPuniversalGB}]
%Note that we are now saturating $J_\scrE$ with respect to all conditional probabilities $p_{i|I}$, not just the sums $p_{I|I}$.  
Let $\mcA_{G(\scrE)}$ be the vertex-edge incidence matrix of $G(\scrE)$.  By Proposition \ref{prop:bipartiteUnimodular},  
its cycle binomials (circuits) give a universal Gr\"obner basis of $I_{\mcA_{G(\scrE)}}$.
In fact, the induced cycles are enough to generate this ideal \cite{AokiTakemura2002}.  Suppose $C$ is a cycle and $e$ a chord, and split $C$ into two cycles $C_1$ and $C_2$, both containing $e$ (but in opposite directions).  Associate cycle binomials $f_{C_1}$ and $f_{C_2}$, respectively. Then the $S$-polynomial (\S \ref{sec:toricvar}) with the $e$-containing terms leading is $f_C$.  However, this is no longer necessarily a Gr\"obner basis.
For example, let $\scrE = \{ \{12\}, \{23\}, \{123\}\}$ as in Figure \ref{fig:bipartite12--23-123}. 
\begin{figure}[htb]
\[ \SelectTips{cm}{10}
   \begin{xy}<1cm,0cm>:
(0,0) *+!{123}="123";
      p+(0,1) *+!{1}="1",
      p+(1,-.86) *+!{3}="3",
      p+(-1,-.86) *+!{2}="2";
(-1,.5) *+!{12}="12"; 
(0,-1.3) *+!{23}="23"; 
{\ar@{->}_{p_{1}} "123"; "1"};
{\ar@{->}^{p_{2}} "123"; "2"};
{\ar@{->}^{p_{3}} "123"; "3"};
{\ar@{->}^{p_{1|12}} "12"; "1"};
{\ar@{->}_{p_{2|12}} "12"; "2"};
{\ar@{->}^{p_{2|23}} "23"; "2"};
{\ar@{->}_{p_{3|23}} "23"; "3"};
   \end{xy}
\]
\caption{Bipartite graph for $\scrE = \{ \{1,2\},\{2,3\},\{1,2,3\}\}$.} \label{fig:bipartite12--23-123}
\end{figure}
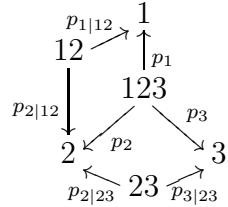

The outer cycle $C=1 \ra 12 \ra 2 \ \ra 23 \ra 3 \ra 123 \ra 1$ gives the cycle binomial $f_C = p_{1|12}p_{2|23} p_{3|123} - p_{2|12} p_{3|23} p_{1|123}$.  The cycle $C$ has a chord $2-123$, and the binomial $f_C$ lies in the ideal of the two binomials 
\[
p_{1|12}p_{2|123} - p_{2|12}p_{1|123} \;\; \text{and} \;\; 
p_{2|23}p_{3|123} - p_{3|23}p_{2|123}
\]
after splitting along the chord. These are both the induced cycles of the graph.  However, for a term order (\S \ref{sec:toricvar}) prioritizing $p_{2|123}$ (e.g. lexicographic with $p_{2|123} \succ \cdots$), the leading term of $f_C$ cannot lie in the initial ideal $\langle p_{1|12}p_{2|123}, p_{3|23}p_{2|123} \rangle$ of the ideal generated by the chordal binomials. 

Next we show that the graph ideal and conditional probability ideal coincide, 
$I_{\mcA_{G(\scrE)}} = I_\scrE$.
For the containment $I_{\mcA_{G(\scrE)}} \supseteq I_\scrE$, first observe that $I_{{\rm Bayes}(\scrE)} \subseteq I_{\mcA_{G(\scrE)}}$.  
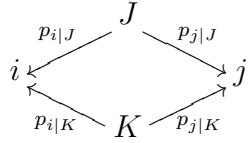
\begin{figure}[htb]
\[ \SelectTips{cm}{10}
   \begin{xy}<1.5cm,0cm>:
(0,0);
      p+(-1,0) *+!{i}="i",
      p+(1,0) *+!{j}="j";
(0,.5) *+!{J}="J"; 
(0,-.5) *+!{K}="K"; 
{\ar@{->}^{p_{i|K}} "K"; "i"};
{\ar@{->}_{p_{i|J}} "J"; "i"};
{\ar@{->}_{p_{j|K}} "K"; "j"};
{\ar@{->}^{p_{j|J}} "J"; "j"};
   \end{xy}
\]
\caption{Subgraph of $G(\scrE)$ giving a Bayes binomial.} \label{fig:BayesSubgraph}
\end{figure}
This is because if $J,K \in \scrE$ with $i,j \in J \subseteq K$, we have the subgraph in Figure \ref{fig:BayesSubgraph}, which is a cycle with associated cycle binomial $\pjJ \piK - \piJ \pjK  $.
Together with Proposition \ref{prop:BayesSat}, we now have 
\[
J_\scrE \subseteq I_{{\rm Bayes}(\scrE)} \subseteq I_{\mcA_{G(\scrE)}}
\]
so, since saturation is inclusion-preserving and $ I_{\mcA_{G(\scrE)}}$ is prime, 
\[
I_\scrE = (J_\scrE :(\alpha_{\scrE} \beta_{\scrE})^{\infty} ) 
        \subseteq  (I_{\mcA_{G(\scrE)}} : (\alpha_{\scrE} \beta_{\scrE})^{\infty} ) 
        = I_{\mcA_{G(\scrE)}}.
\]

Now we show the reverse inclusion $I_{\mcA_{G(\scrE)}} \subseteq I_\scrE$.  Again by Proposition \ref{prop:BayesSat}, we have
\[
 I_{{\rm Bayes}(\scrE)} \subseteq I_\scrE .
\]
Now assume that $[m] \in \scrE$, so that $p_1, \dots, p_m \in \C[\scrE]$.  
We claim that in fact $I_{\mcA_{G(\scrE)}} \subseteq (I_{{\rm Bayes}(\scrE)}: \prod_{i=1}^{m} p_{i} ) $, from which the result will follow.  Let $C$ be an induced cycle of $G(\scrE)$, and $f_C$ its cycle binomial.  We must show that this cycle binomial can be obtained from the Bayes binomials, up to multiplication by $ \prod_{i=1}^{m} p_{i}$.  Let $C$ be the cycle $$i_1 \leftarrow J_1 \ra i_2 \leftarrow J_2 \rightarrow \cdots \rightarrow i_k \leftarrow J_k \rightarrow i_1.$$  With this notation we have $i_1, i_2 \in J_1$, $i_2, i_3 \in J_2$, $\dots$, $i_1, i_k \in J_k$.  Then
\[
f_C  = p_{i_2 | J_1} p_{i_3 | J_2} \cdots p_{i_k | J_{k-1}} p_{i_1 | J_k}
       -  p_{i_1 | J_1} p_{i_2 | J_2} \cdots p_{i_k | J_k} .
\] 
We show the first monomial of $( \prod_{i=1}^{k} p_{i} ) f_C$ is equal to the second mod $ I_{{\rm Bayes}(\scrE)}$.  Pair off as follows:
\begin{eqnarray*}
\phantom{=} &  (\underline{p_{i_1}} p_{i_2} p_{i_3}  \cdots p_{i_k}) \underline{p_{i_2 | J_1}} p_{i_3 | J_2} p_{i_4 | J_3} \cdots p_{i_k | J_{k-1}} p_{i_1 | J_k} & \text{Step 1} \\
 = & (p_{i_2} \underline{p_{i_2}}  p_{i_3} \cdots p_{i_k}) p_{i_1 | J_1} \underline{p_{i_3 | J_2}}  p_{i_4 | J_3} \cdots p_{i_k | J_{k-1}} p_{i_1 | J_k} & \text{Step 2}\\
 = & (p_{i_2}  p_{i_3}  \underline{p_{i_3}} \cdots p_{i_k}) p_{i_1 | J_1} p_{i_2 | J_2} \underline{ p_{i_4 | J_3}} \cdots p_{i_k | J_{k-1}} p_{i_1 | J_k} & \text{Step 3} \\
& \vdots & 
\end{eqnarray*}
where the equalities hold mod  $ I_{{\rm Bayes}(\scrE)}$.  Continuing in this fashion, at step $k-1$ we have
\begin{eqnarray*}
 = & (p_{i_2}  p_{i_3} \cdots  p_{i_{k-1}} \underline{p_{i_{k-1}}} p_{i_k}) p_{i_1 | J_1} p_{i_2 | J_2} \cdots p_{i_{k-2} | J_{k-2}} \underline{p_{i_k | J_{k-1}}} p_{i_1 | J_{k}} & \text{Step $k-1$} \\
 = & (p_{i_2}  p_{i_3} \cdots  p_{i_{k-1}} p_{i_{k}} \underline{p_{i_k}}) p_{i_1 | J_1} p_{i_2 | J_2} \cdots p_{i_{k-2} | J_{k-2}} p_{i_{k-1} | J_{k-1}} \underline{p_{i_1 | J_{k}}} & \text{Step $k$} \\
 = & (p_{i_2}  p_{i_3} \cdots  p_{i_{k-1}} p_{i_{k}} p_{i_1}) p_{i_1 | J_1} p_{i_2 | J_2} \cdots p_{i_{k-2} | J_{k-2}} p_{i_{k-1} | J_{k-1}} p_{i_k | J_{k}} & \text{Step $k+1$} \\
\end{eqnarray*}
as desired.  In terms of $G(\scrE)$, this amounts to breaking up a long cycle into 4-cycles passing through $[m]$, and erasing the overlaps among these cycles.   Thus since the induced cycles generate $I_{\mcA_{G(\scrE)}}$, we have 
\[
I_{\mcA_{G(\scrE)}} \subseteq ( I_{{\rm Bayes}(\scrE)} : \prod_{i=1}^m p_i)) \subseteq I_\scrE
\]
This proves the result in the special case $[m] \in \scrE$.  In the general case, 
suppose we have some $\scrE$ not containing $[m]$, enabling us to obtain relations among 'pure' conditional probabilities (i.e. excluding $p_1, \dots, p_m$).  Let $\scrE' = \scrE \union [m]$ and apply the special case of the Theorem.  Then by \cite[Proposition 4.13(c)]{GBCP}, since we have a universal Gr\"obner basis, we just intersect it with the smaller coordinate ring to obtain a universal Gr\"obner basis of the smaller ring.  This corresponds here to removing the set $[m]$ from $\scrE$ and taking the cycle binomials as our new Gr\"obner basis.  
\end{proof}

\section{Conditional probability and the moment map} \label{sec:CPmoment}
%Show how to project $\mcA$ to get permutohedra and simplices.  Show hypercube projection, fiber polytope.
In this section we show how to recover and generalize some results of Mat\'u\v{s} \cite{Matus2003Conditional} using toric geometry.  
The main result we will expand upon maps the space of conditional probability distributions (Definition \ref{def:CPdist}) for all possible conditioned events $\scrE = \{ I \subset [m]: |I| \geq 2\}$ onto the permutohedron by first projecting down to events of size 2, $\scrE=\{I \subset [m] : |I|=2\}$.

\begin{thm}[Mat\'u\v{s} \cite{Matus2003Conditional}] \label{thm:Matusthm}
For $\scrE = \{ I \subset [m] : |I| \geq 2 \}$ and $\mathbf{p}$ a conditional probability distribution (Definition \ref{def:CPdist}), the map $W:\R^{\|\scrE\|} \ra \R^m$, given by 
$$
W_i(\mathbf{p})  =  \sum_{j \in [m] \setminus i} p_{i|ij},
$$
restricts to a homeomorphism of the space of conditional probabilities onto the $m-1$ dimensional permutohedron \index{permutohedron} $\mathbf{P}_{m-1}$.
\end{thm}

Note that the linear map $W$ is the restriction of $\mcA = \mcA_{G(\scrE)}$ to the rows labeled by the vertex set $V$ in $G$ \eqref{eq:vertexsets} and to the columns labeled by two-event conditional probabilities (edges in $G(\scrE)$) $p_{i|ij}$.  In fact $\mcA$, will in general define a map from the space of projective conditional probability distributions onto a generalized permutohedron $\Delta_{\scrE}$ defined below.%  This polytope is the Minkowski sum of simplices (MSS) defined in Section \ref{sec:CRTSubmod} for the subset $\scrE$ of $2^{[m]}$ and corresponding set of simplices $\{\Delta_I, I \in \scrE\}$.

First consider the multiprojective toric  variety $\mcZ_\mcA$ cut out of $\prod_{I \in \scrE} \PP^{|I|-1}$ by the equations of Theorem \ref{thm:CPuniversalGB}, i.e.~the space of projective conditional probability distributions.  
In Section \ref{sec:toricvar} we recall the definition of the affine toric variety $X_\mcA$ associated to an integer matrix $\mcA$, and the projective  toric variety $Y_\mcA$ associated to a $\Z$-graded matrix $\mcA$ (that is, a matrix $\mcA$ such that $(1,1, \dots, 1)$ lies in its rowspan).  Given a matrix $\mcA = \mcA_{G(\scrE)}$, the space of $\scrE$-projective conditional probability distributions $\mcZ_\mcA$ is the closure of the image of the map $\mathbf{f}_\mcA:\theta \mapsto \theta^\mcA$, viewed as an element of $\prod_{I \in \scrE} \PP^{|I|-1}$.  Equipping this product space with multihomogeneous coordinates $((p_{i_1 | I}: \cdots : p_{i_{|I|} |I}), I \in \scrE)$, the variety $\mcZ_\mcA$ is cut out by the (multihomogeneous) toric ideal $I_\mcA$.
Suppose that we have $\union_{I \in \scrE} =[m]$.  Then because we view the points $((p_{i_1 | I}: \cdots : p_{i_{|I|} |I}), I \in \scrE)$ as elements of $\prod_{I \in \scrE} \PP^{|I|-1}$, the dimension of this variety is $m-1$ as expected, though the rank of $\mcA$ is larger. %requires pf

We now develop a version of the moment map of toric geometry applicable to the variety of projective conditional probability distributions.
Hereafter we index the columns of $\mcA$ by the conditional probability they represent, i.e. $\mcA = ( a_{\cdot i|I} : i \in  I \in \scrE )$. 
We will require a multigraded notion to play the role of the convex hull $\conv(\mcA)$ in the moment map.  We define
\[
\mconv(\mcA) = \{\sum_{I \in \scrE} \sum_{j \in I} \lambda_{j|I} a_{ \cdot j|I} : \lambda_{j|I} \in \R_{\geq 0}, \; \sum_{j \in I} \lambda_{j|I} = 1  \}.
\]
%Observe that $\mconv(\mcA)$ is the Minkowski sum over $I \in \scrE$ of the convex hulls of the submatrices $\mcA_{V,I}$.  
A function $\, w : 2^{[n]} \rightarrow \R \, $ is called {\em submodular} if
$\,w(I) + w(J)\, \geq\, w(I \cap J) +  w(I \cup J)\,$ for $I,J \subseteq [n]$. 
Each subset $I$ of $[m]$ defines a submodular function $w_I$ on $2^{[n]}$ 
by setting $w_I (J) = 1$ if $I \cap J $ is non-empty
and  $w_I(J) = 0$ if $I \cap J $ is empty for $J \in 2^{[n]}$.
The function $w$ defines a convex polytope $Q_w$ of dimension $\leq n-1$
as follows:
%$Q_{w_I}$
\begin{eqnarray*}  Q_w \,\,\, :=  &
\bigl\{ \, x \in \R^n \,: \,
x_1 + x_2 + \cdots + x_n = w([n]) \\
& \text{\  \,and } \sum\nolimits_{i \in I} x_i \leq w(J)\,\,
\hbox{for all} \,\, \emptyset\neq J \subseteq [n]   \,\bigr\}
\end{eqnarray*}
Thus the polytope corresponding to a subset $I$ is the simplex
$\Delta_I = {\rm conv} \{ e_k :  k \in I \}$.
Now consider an arbitrary subset  $\,\scrE = \{I_1,I_2,\ldots,I_r \}\,$
of $2^{[m]}$. It defines the submodular function
$\,w_{\scrE} =  w_{I_1} + w_{I_2} + \cdots + w_{I_r}$.
 The corresponding polytope $Q_{w_\scrE}$ is now the Minkowski sum \cite{Ziegler1995}
\begin{equation} \label{eq:MSS}
 \Delta_\scrE \quad = \quad \Delta_{I_1} + \Delta_{I_2} + \cdots + \Delta_{I_r}. 
\end{equation}

\begin{prop}
The projection of $\mconv(\mcA_{G(\scrE)})$ to the $V$-coordinates  \eqref{eq:vertexsets} is $\Delta_\scrE$.
\end{prop}
\begin{proof}
The $\mconv$ construction is equivalent to translating each simplex that is the convex hull of each set of vectors $\mcA_I \subset \mcA$ by setting its $U$-coordinates  \eqref{eq:vertexsets} all to $1$, then taking the Minkowski sum. 
\end{proof}
Next is a version of Theorem \ref{thm:MomentMap} for varieties $\mcZ_\mcA$.  Note that $|V|=m$ when $\union_{I \in \scrE} I = [m]$.  Now we have a separate partition function for each conditioned-upon set.

\begin{thm} \label{thm:MultiMomentMap}
For $\mcA=\mcA_{G(\scrE)}$, the map $\nu: \mcZ_\mcA \ra \R^{|V|}$ defined by 
\[
\nu(z) = \sum_{I \in \scrE} \frac{1}{Z_I(z)} \sum_{i \in I} |z_{i|I}| a_{\cdot i|I},
\]
where $Z_I = \sum_{i \in I} |z_{i|I}|$, 
maps $\mcZ_\mcA$ onto $\mconv(\mcA)$, and is a bijection on $\mcZ_{\mcA, \geq 0}$.
\end{thm}

\begin{proof}
The map $\nu$ is the composition of two maps.  The first map, $\nu_1: \mcZ_\mcA \ra \prod_{I \in \scrE} \Delta_I$, is a product of maps $\mu_{1}$ corresponding to each submatrix $\mcA_I$ as in the proof of Theorem \ref{thm:MomentMap}.  It ssends a point $((z_{i_1|I}, \dots, z_{i_{|I|}|I} ), I \in \scrE) \in \mcZ_\mcA$ to the point $\bfp=(p_{i|I} = \frac{1}{Z_I(z)} |z_{i|I}| : i \in I \in \scrE)$ in the product of simplices $\prod_{I \in \scrE} \Delta_I$, which can be thought of as possibly redundant barycentric coordinates.  The second map, $\nu_2$, corresponds to the Minkowski sum, with $\nu_2: \prod_{I \in \scrE} \Delta_I \ra \mconv(\mcA)$ sending $\bfp$ to $\mcA \bfp$.  
Whereas in the simplex case (and for a single $\mcA_I$) in Theorem \ref{thm:MomentMap}, $\mu_1$ and $\mu_2$ are identities, here there is additional ambiguity introduced by the Minkowski sum.  In particular, let $b \in \Delta_\scrE$ (\ref{eq:MSS}).  Then the preimage of $b$ in $ \prod_{I \in \scrE} \Delta_I$ is 
\[
P_\mcA(b) = \{\bfp: \mcA \bfp = b  \} \cap \prod_{I \in \scrE} \Delta_I,
\]
and in general consists of a polytope.  This is illustrated in Figure \ref{fig:MinkSumAmbig}, where the polytope $P_\mcA (b)$ is the set of pairs of points in the first and second simplex that add to $b$.
Analogously to the one-factor case (Theorem \ref{thm:MomentMap}), we will choose among the points of this fiber by selecting the maximum entropy point (or the point closest in the KL-divergence sense to the point representing a uniform distribution in all simplices).  The resulting space of solutions (the space of conditional probability distributions) is illustrated in Figure \ref{fig:TobleroneBlowup}.
 %In the example shown in Figure \ref{fig:MinkSumAmbig}, the intersection of the variety $\mcZ_\mcA$ with the product of simplices is the blow-up of $\PP^2$ at a labeled point of Figure \ref{fig:projprobsimplex} intersected with a triangular prism. 

Setting $D(\bfp) = D(\bfp ||\bfp^{uniform})$ so \[
D(\bfp) = \sum_{i \in I \in \scrE} p_{i|I} \log \piI - \sum_{i  \in I \in \scrE} \piI \log (\frac{1}{|I|}), 
\]
the Hessian of $D$ is $\frac{1}{p_{i|I}}$ on the diagonal and zero elsewhere.  Thus it is positive definite on the interior of $\prod_{I \in \scrE} \Delta_I$, and on points of the relative interior after restricting to nonzero coordinates.  Thus $D$ has a unique minimum $\bfp^*$ on $\prod_{I \in \scrE} \Delta_I$.  Were there another minimum, the (possibly restricted) Hessian would be positive definite on the open segment connecting it with $\bfp^*$.  We now argue that $\bfp^* \in \mcZ_\mcA$.

First suppose $\bfp^* \in (\prod_{I \in \scrE} \Delta_I)^\circ$, so that $0 < \piI <1$ in all coordinates, and let $u \in \ker \mcA$.  We must show that $p^{u^+} = p^{u^-}$.  For small $t$, $\bfp^* + tu \in \prod_{I \in \scrE} \Delta_I$ and \[
D(\bfp^* + tu) = \sum_{i \in I \in \scrE} (\piI + tu_{i|I}) \log  (\piI + tu_{i|I}) - \sum_{i \in I \in \scrE} (\piI + tu_{i|I}) \log \frac{1}{|I|}
\]\[
\frac{dD}{dt} = \sum_{i \in I \in \scrE}  u_{i|I} \log (\piI + tu_{i|I}) + \sum_{i \in I \in \scrE} u_{i|I} + \sum_{i \in I \in \scrE} u_{i|I} \log\frac{1}{|I|}.
\]
Since $\mcA$ is $\scrE$-multigraded, the last two terms of $\frac{dD}{dt}$ are zero (i.e. $(1,1, \dots, 1) \in \R^{\|\scrE\|}$ is in the rowspace of $\mcA$, and $(1,1,\dots, 1) \in \R^{|I|}$ is in the rowspace of each $\mcA_I$).  At $t=0$, the first order condition implies that 
\[
0 = \frac{dD}{dt} = \sum_{i \in 
I \in \scrE} u_{i|I} \log \piI.
\]
Grouping the sum by the sign of $u_{i|I}$ and changing to exponential notation, \begin{equation} \label{eq:mgmmbin}
p^{u^+} = p^{u^-}
\end{equation}
as desired.

Now suppose that $\bfp^*$ lies on the boundary of $\prod_{I \in \scrE} \Delta_I$. 
If the zeros of $\bfp$ lie outside $\supp(u)$, the argument made above for $\bfp^*$ in the interior holds after extending $D$ with the limit $p \log(p) \ra 0$ as $p \ra 0$.
If there are zeros on both sides of (\ref{eq:mgmmbin}), i.e. $p_{i|I}=0=\pjJ$ for indices $i|I \in \supp(u^+)$ and $j|J \in \supp(u^-)$, then the relation holds with $0=0$.  

We may assume $\piI = 0$ for some index $i|I \in \supp(u^+)$ in considering the  two remaining cases.  The first case has $\pjJ = 1$ for some index $j|J \in \supp(u^+)$.  Because of the multigrading of $\mcA$, which requires  for any $J \in \scrE$ and $u \in \ker \mcA$ that $\sum_{j \in J}u_{j|J}=0$, it must be that there exists $k|J \in \supp(u^-)$.  Then since $\bfp \in \prod_{I \in \scrE}\Delta_I$, we have $p_{k|J} = 0$ and the relation (\ref{eq:mgmmbin}) holds as $0=0$.  

The second case has $0 \leq \pjJ <1$ for all $j|J \in \supp(u^+)$ and $0< \pkK \leq 1$ for all $k|K \in \supp(u^-)$.  Then for small $t$, $\bfp^* + tu \in P_\mcA(b)$. %The $p_{u^+}$ may fall a little and $p_{u^-}$ increase a little
Then we have 
\begin{equation} \label{eq:mpmm}
\frac{dD}{dt}  = \sum_{\{i|I \, : \, \piI = 0\}} u_{i|I} (\piI + t u_{i|I}) + \sum_{\{j|J \, :  \, \pjJ \neq 0\}}  u_{j|J} (\pjJ + t u_{j|J}) .
\end{equation}
Then the first term on the right hand side of (\ref{eq:mpmm}) approaches negative infinity as $t \ra 0$ while the second approaches a constant; this contradicts the optimality of $\bfp^*$, so this case cannot arise.
\end{proof}
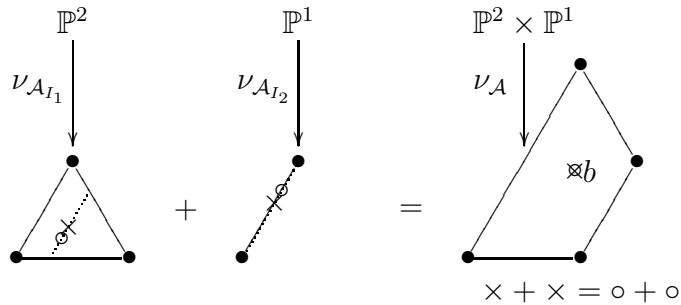
\begin{figure}[htb]
\[
\begin{xy}<1.5cm,0cm>:
%triangle
(0,0) *{\bullet};
      p+(.5,0.86) *{\bullet} **@{-};
      p+(.5,-0.86) *{\bullet} **@{-};
      p+(-1,0) **@{-};
(1.5,0.43) *{+};
%line segment
(2,0) *{\bullet};
      p+(.5,0.86) *{\bullet} **@{-};
(3.5,0.43) *{=};
%triangle+line segement
(4,0) *{\bullet};
      p+(1,1.72) *{\bullet} **@{-};
      p+(.5,-0.86) *{\bullet} **@{-};
      p+(-.5,-0.86) *{\bullet} **@{-};
      p+(-1,0) **@{-};
%INNER POINTS
%inner line segment in triangle with two marked points
(.3,0);
      p+(.325,0.559)  **@{.}, %.65d
      p+(.1,0.172) *{\circ}, %.2d
      p+(.1625,0.2795) *{\times}; %(.65/2)d
%marked points in line segment
(2,0);
      p+(.5,0.86)  **@{.},
      p+(.2875,0.4945) *{\times}, %(.9 - .65/2)d 
      p+(.35,0.602) *{\circ};  %.7d
%marked point in sum
(4.5,0);
      p+(.45,.774) *{\times} *{\circ} *+!L{b}; %.9d
(5,-.3) *{\times + \times = \circ + \circ};
% Some geometry
(.5,2) *+!{\PP^2};
      p+(0,-1)  **@{-} *@{>},
      p+(-.3,-.5) *+!{\nu_{\mcA_{I_1}}};
(2.5,2) *+!{\PP^1};
      p+(0,-1)  **@{-} *@{>},
      p+(-.3,-.5) *+!{\nu_{\mcA_{I_2}}};
(4.5,2) *+!{\PP^2 \times \PP^1};
      p+(0,-1)  **@{-} *@{>},
      p+(-.3,-.5) *+!{\nu_\mcA};
\end{xy}
\]
\caption{Ambiguity arising from Minkowski sum of simplices: two points appearing in the fiber over $b$ in $\prod_{I \in \scrE} \Delta_I$.  For any point on the dotted line, there is a point in the second simplex such that their sum is $b$.  We choose $\times$ among these points by maximizing entropy in the conditional probability distribution. See Figure \ref{fig:TobleroneBlowup} for the space of solutions.
} \label{fig:MinkSumAmbig}
\end{figure}

\begin{figure}[htb]
\[
\SelectTips{cm}{10}
\begin{xy}<2cm,0cm>:
%Foreground triangle
(2,0) ="F3" *+!R{} *{\bullet};
      p+(.5,\sinpioverthree)  ="F1" *+!D{} *{\bullet} **@{-};
      p+(.5,-.86)  ="F2" *+!L{} *{\bullet} **@{-};
      p+(-1,0) **@{-};
%Background Triangle
(2.7,.7)  ="B3" *+!R{} *{\bullet};
      p+(.5,\sinpioverthree)  ="B1" *+!D{} *{\bullet} **@{-};
      p+(.5,-.86)  ="B2" *+!L{} *{\bullet} **@{-};
      p+(-1,0) **@{-};
%Way Background Triangle
(3.4,1.4) *+!DR{} *{\bullet};
      p+(.5,\sinpioverthree) *+!L{(1:0:0)} *{\bullet} **@{-};
        p+(.5,\sinpioverthree) **@{-},
        p+(-1,-1.72)  **@{-},
      p+(.5,-.86) *+!DL{(0:1:0)} *{\bullet} **@{-};
        p+(.25,-.43) **@{-},
        p+(-1,1.72) **@{-},
      p+(-1,0) **@{-};
        p+(-1,0) **@{-},
        p+(2,0) **@{-},
%Connect foreground and background triangle
{\ar@{-}^{E} "F1"; "B1"};
"F1";"B1" **@{-};
"F2";"B2" **@{-};
"F3";"B3" **@{-};
%Extra points on triangles
"F3";
      p+(.1,.172)  ="Fe1";
      p+(.1,.172)  ="Fe2";
      p+(.1,.172)  ="Fe3";
      p+(.1,.172)  ="Fe4";
"B1";
      p+(.1,-.172)  ="Be1";
      p+(.1,-.172)  ="Be2";
      p+(.1,-.172)  ="Be3";
      p+(.1,-.172)  ="Be4";
%Edges between the extra points
"Fe1";"Be4" **@{--};
"Fe2";"Be3" **@{--};
"Fe3";"Be2" **@{--};
"Fe4";"Be1" **@{--};
"F3";"B2" **@{--};
"B2";
      p+(1.1,-.1)="trap1";
"trap1";
      p+(-.25,-.43) **@{-};
      p+(1,0) **@{-};
      p+(-.25,.43) **@{-};
      p+(-.5,0) **@{-};
"B2";
      p+(.1,-.2) = "StartArrow",
      p+(.9,-.3) = "EndArrow";
{\ar@{->}^{\nu} "StartArrow"; "EndArrow"};
\end{xy}
\]
\caption{The space of conditional probability distributions is the blow-up of $\PP^2$ at the point $p_2=p_3=0$ of Figure \ref{fig:projprobsimplex}, intersected with a triangular prism.  In general and in higher dimensions, blow-ups are along the conditioned-upon faces.  $E$ has homogeneous coordinates $(p_{2|23}:p_{3|23})$ and the triangle has homogeneous coordinates $(p_1:p_2:p_3)$.
} \label{fig:TobleroneBlowup}
\end{figure}

We now give a couple of examples.

\begin{example}
For the case $m=3$ with $\scrE = \{ 12,13,23,123\}$, the matrix $\mcA$ is
\[
\bordermatrix{
& p_1 & p_2 & p_3 & p_{1|12} & p_{2|12} & p_{1|13} & p_{3|13} & p_{2|23} & p_{3|23}\cr
1   & 1& 0& 0& 1& 0& 1& 0& 0& 0\cr
2   & 0& 1& 0& 0& 1& 0& 0& 1& 0\cr
3   & 0& 0& 1& 0& 0& 0& 1& 0& 1\cr
% \hline 
12  & 0& 0& 0& 1& 1& 0& 0& 0& 0\cr
13  & 0& 0& 0& 0& 0& 1& 1& 0& 0\cr
23  & 0& 0& 0& 0& 0& 0& 0& 1& 1\cr
123 & 1& 1& 1& 0& 0& 0& 0& 0& 0\cr
}
\]
The $U$-coordinate rows are labeled $1,2,3$ and the $V$-coordinate rows are labeled $12,13,23, 123$.  The polytope $\mconv(\mcA)$ is the permutohedron which is the convex hull of the permutations of $(3,1,0)$, shown in Figure \ref{fig:superman}.
\begin{figure}[htb]
\[
   \begin{xy}<1cm,0cm>:
% 1 direction is (0,1)
% 2 direction is (.86, -.5)
% 3 direction is (-.86, -.5)
(-.86,2.5)  *{\bullet} ="123"   *+!RD{301};
(.86, 2.5)   *{\bullet} ="132"   *+!LD{310};
(2.58,-.5)  *{\bullet} ="312"   *+!L{130};
(1.72,-2)    *{\bullet} ="321"   *+!LU{031};
(-1.72,-2)   *{\bullet} ="231"   *+!RU{013};
(-2.58,-.5) *{\bullet} ="213"   *+!R{103};
%face points
(1.72,1) *{\bullet} *+!R{220};
(-1.72,1) *{\bullet} *+!L{202};
(0,-2) *{\bullet} *+!D{022};
%interior points
(0,1) *{\bullet} *+!U{211};
(0.86,-.5) *{\bullet} *+!DR{121};
(-.86,-.5) *{\bullet} *+!DL{112};
   "123";"132" **@{-};
   "132";"312" **@{-};
   "312";"321" **@{-};
   "321";"231" **@{-};
   "231";"213" **@{-};
   "213";"123" **@{-};
   \end{xy}
\]
\caption{Multigraded convex hull of $\mcA$ for $n=3$ and $\scrE=\{I\subseteq [n] : |I| \geq 2\}$.  The last four coordinates, not shown, are all 1.} \label{fig:superman}
\end{figure}
Letting $\mcA'$ be the last six columns of $\mcA$ (restriction to $\{I \subseteq [n]: |I|=2\}$), $\mconv(\mcA')$ is the regular permutohedron $\conv((2,1,0)$, $(2,0,1)$, $(1,0,2)$, $(1,2,0)$, $(0,2,1)$, $(0,1,2))$, lifted with the last four coordinates all $1$.  This is  illustrated in Figure \ref{fig:regularmconv}.
\begin{figure}[htb]
\[
\begin{xy}<10mm,0cm>:
(-.5,\sinpioverthree)  *{\bullet} ="123"   *+!RD{201};
(.5 ,\sinpioverthree)  *{\bullet} ="132"   *+!LD{210};
(1,0)                  *{\bullet} ="312"   *+!L{120};
(.5,-\sinpioverthree)  *{\bullet} ="321"   *+!LU{021};
(-.5,-\sinpioverthree) *{\bullet} ="231"   *+!RU{012};
(-1,0)                 *{\bullet} ="213"   *+!R{102};
   "123";"132" **@{-};
   "132";"312" **@{-};
   "312";"321" **@{-};
   "321";"231" **@{-};
   "231";"213" **@{-};
   "213";"123" **@{-};
\end{xy}
\]
\caption{Multigraded convex hull of $\mcA$ for $n=3$ and $\scrE=\{I\subseteq [n] : |I|=2\}$. The last four coordinates, not shown, are all $1$} \label{fig:regularmconv}
\end{figure}

\end{example}

The theorem of Mat\'u{\v s} (Theorem \ref{thm:Matusthm}) works in this way by projecting first from $\scrE = \{I:|I|\geq 2 \}$ to  $\scrE = \{I:|I| = 2 \}$ as in Figure \ref{fig:regularmconv}.  Thus the result may be understood as saying that instead of all simplices, we can obtain a regular permutohedron merely as the zonotope given by the Minkowski sum of the $1$-simplices.

\section{Partially observed discrete random variables} \label{sec:caseofrv}

Let $X_1, \dots, X_n$ be discrete random variables with $X_i$ taking values $x_i^1, \dots, x_i^{d_i}$.  Then the $m=\prod_{i=1}^n d_i$ singleton events in $\Omega$ are the elements of the Cartesian product of the sets of states which each random variable may assume.  
For a subset of random variables $X_{i_1}, \dots, X_{i_k}$ with $S:=\{i_1, \dots, i_k\} \subseteq [n]$, we write $\Omega_S$ for the Cartesian product of the states of this subset of the random variables.  We also denote by $x|_S$ the restriction of some global state $x \in \Omega$ to the states of the random variables in $S$.
 Then the set of events $\scrE$ has the form:
\begin{equation} \label{eq:scrErv}
\scrE = \{ x' \in \Omega \; : \; x'|_S = x_S \text{ for some } S \subseteq [n], x_S \in \Omega_S \}
\end{equation}

Let $E(x_S)$ denote the event which is the union of all singleton events with random variables $S$ in state $x_S$.  For example, let $n=3$, $d_i =2$ with states denoted $0$ and $1$, and $S=\{1,3\}$.  Then $E(x^0_1 x^1_3) = \{0010,0011,0110,0111\}$, which corresponds to a 2-face of the 4-cube.  Now we may write with the more usual notation
\[
p_{x_A | x_B} := p_{E(x_A) \cap E(x_B) | E(x_B)}
\]
which is convenient for considering, say, the conditional probability of having a disease given a positive test result.  
Besag's relation (\ref{eq:BesagRelation}) among positive conditional probabilities is written this way:
\begin{equation} 
\frac{P(\mathbf{x})}{P(\mathbf{y})} = \prod_{i=1}^n \frac{P(x_i|x_1, \dots, x_{i-1}, y_{i+1}, \dots, y_n)}{P(y_i|x_1, \dots, x_{i-1}, y_{i+1}, \dots, y_n)}.
\end{equation}
This is a special case of the relations derived in Theorem \ref{thm:CPuniversalGB}, as we now explain.

Denote the event $x_1, \dots, x_{j-1}, y_{j}, \dots, y_n$ by $j$, so the singleton events are $(y_1, \dots, y_n)= 1,2, \dots, n+1 =(x_1, \dots, x_n)$.  The set $\scrE$ consists of the event $\{1, \dots, n+1\}$ together with the events $\{j,j+1\}$ for $j=1, \dots, n$.  Then the cleared-denominator version of (\ref{eq:BesagRelation}) is the outer cycle $[n+1] \ra 1 \leftarrow 12 \ra 2 \leftarrow \cdots \leftarrow n,n+1 \ra n+1 \leftarrow [n+1]$ in the graph $G_{\scrE}$.  For example, with three variables we have events $1=(y_1,y_2,y_3)$, $2=(x_1,y_2,y_3)$, $3=(x_1,x_2,y_3)$, and $4=(x_1,x_2,x_3)$.  The relation  (\ref{eq:BesagRelation}) is \[
\frac{p_4}{p_1} = \frac{p_{2|12} p_{3|23}p_{4|34}}{p_{1|12} p_{2|23} p_{3|34} },
\]
corresponding to the cycle binomial 
\[
p_1 p_{2|12} p_{3|23}p_{4|34} - p_4 p_{1|12} p_{2|23} p_{3|34}, 
\]
which is $f_C$ for the outer cycle $C$ of the graph in Figure \ref{fig:bipartite12-23-34-1234}.

\begin{figure}[htb]
\[ \SelectTips{cm}{10}
   \begin{xy}<1.2cm,0cm>:
(0,0) *+!{1234}="1234";
      p+(-1,1) *+!{1}="1",
      p+(1,1) *+!{2}="2",
      p+(1,-1) *+!{3}="3",
      p+(-1,-1) *+!{4}="4",
      p+(0,1) *+!{12}="12",
      p+(1,0) *+!{23}="23",
      p+(0,-1) *+!{34}="34",
{\ar@{->}^{p_{1}} "1234"; "1"};
{\ar@{->}^{p_{2}} "1234"; "2"};
{\ar@{->}_{p_{3}} "1234"; "3"};
{\ar@{->}_{p_{4}} "1234"; "4"};
{\ar@{->}_{p_{1|12}} "12"; "1"};
{\ar@{->}^{p_{2|12}} "12"; "2"};
{\ar@{->}_{p_{2|23}} "23"; "2"};
{\ar@{->}^{p_{3|23}} "23"; "3"};
{\ar@{->}_{p_{3|34}} "34"; "3"};
{\ar@{->}^{p_{4|34}} "34"; "4"};
   \end{xy}
\]
\caption{Bipartite graph for $\scrE = \{ \{1,2\},\{2,3\},\{3,4\},\{1,2,3,4\}\}$.} \label{fig:bipartite12-23-34-1234}
\end{figure}
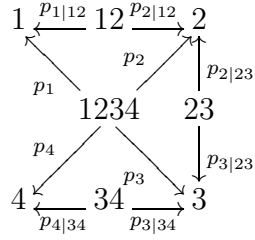

\section{Bayes' rule} \label{sec:BayesExplanation}
Because of the Bayes binomials, on points which are projective conditional probability distributions, we have, with $i,j \subseteq J \subseteq K \subseteq [m]$,
\[ \piK \pjJ = \pjK \piJ . \]
This implies, by summing over $j \in J$, that
\begin{equation} \label{eq:summedBayesBin}
\piK \pJJ = \pJK \piJ . 
\end{equation}
Using two copies of (\ref{eq:summedBayesBin}) with different intermediate sets $J_1$ and $J_2$, we have 
\[
(p_{i|J_1} p_{J_1 | K}) p_{J_2| J_2} = p_{i|K} p_{J_1 | J_1} p_{J_2 | J_2} = (p_{i | J_2} p_{J_2 | K}) p_{J_1 | J_1} %path-indep
\]
which gives a multihomogeneous version of Bayes' rule.  Because we consider the point representing a projective conditional probability distribution as an element of $((p_{i_1 | I}: \cdots : p_{i_{|I|} |I}), I \in \scrE)$, we may set  $ p_{J_1| J_1}$ and $ p_{J_2| J_2}$ to $1$ on an open set containing all probabilistically relevant points, and summing over $i \in I$, this becomes
\[
p_{I|J_1} p_{J_1 | K} =  p_{I | J_2} p_{J_2 | K}.
\]
Or when $ p_{J_1 | K} \neq 0$, 
\[
p_{I|J_1} = \frac{  p_{I | J_2} p_{J_2 | K}}{ p_{J_1 | K}}
\]
so that in particular, with $A, B \subseteq [m]$, and setting $I = A \cap B$, $J_1 = B$, $J_2 = A$, and $K=[m]$ we have the familiar expression for Bayes' rule
\[
p_{A \cap B | B} = \frac{p_{A \cap B | A} p_{A}}{p_B}.
\]

 % on p26 of Blue Collar 2

\section{Appendix: Toric ideals and toric varieties} \label{sec:toricvar}
Here we collect some needed facts about toric ideals and toric varieties based primarily on Sturmfels' book \cite{GBCP}, also referring to \cite{CoxResolutions, Ewald, Fulton1993, MillerSturmfels, ASCB}.
\subsection{Affine toric varieties}
Let $\mcA$ be a $d \times m$ integer matrix, with columns $a_{\cdot, 1} , \dots, a_{\cdot m}$.  Let $\C[x_1, \dots, x_m]$ be a polynomial ring in $m$ variables, and for $u \in \Z^m$ let $x^u = \prod_{j=1}^m x_j^{u_j}$.  The matrix $\mcA$ defines a {\em toric ideal}
\[
I_{\mcA}  = \langle x^{u^+} - x^{u^-} : u \in \ker \mcA \cap \Z^m \rangle ,
\]
where $u^+$ is the positive part of $u$ and $u^-$ the negative.  
The toric ideal $I_\mcA$ is a prime ideal.  A minimal set of binomials which
generates $I_{\mathcal{A}}$ is said to be a {\em Markov basis} for
the matrix $\mathcal{A}$. 
A term order is a total order on the monomials of a polynomial ring such that $1$ is the unique minimal element and $m_1 \succ m_2$ implies $m_3 m_1 \succ m_3 m_2$ for any monomials $m_1, m_2, m_3$.  This order defines the initial monomial of any polynomial, and the initial ideal of an ideal $I$ is generated by the initial monomials $\initial_{\succ} f$ for all $f \in I$.  
A {\em Gr\"obner basis} $\{f_1, \dots, f_k\}$ for an ideal $I$ with respect to a monomial term order $\succ$ has $\initial_{\succ}(I) = \langle \initial_{\succ}(f_1), \dots \initial_{\succ}(f_k) \rangle$.  A Gr\"obner basis is {\em universal} \index{Gr\"obner basis!universal} if it is a Gr\"obner basis for all term orders $\succ$.  For polynomials $f$ and $g$ and term order $\succ$, let $m(f,g)$ be the least common multiple of their leading monomials, and let $f_0$, $g_0$ be their leading terms.  Then their $S$-polynomial is $\frac{m(f,g)}{f_0}f - \frac{m(f,g)}{g_0}g$ and is used in Buchberger's algorithm.

In the affine space $\C^m$ with coordinates $x_1, \dots, x_m$, the ideal $I_\mcA$ cuts out the affine toric variety $X_\mcA$.  The $\R_{\geq 0}$-span of the columns of $\mcA$ define a cone $\pos(\mcA)$, and the $\N$-span defines a semigroup $\N \mcA$.  The corresponding semigroup ring $\C[\N \mcA]$ is isomorphic to the affine coordinate ring $\C[x_1, \dots, x_m]/I_\mcA$, i.e.$X_\mcA \isom \Spec (\C[x_1, \dots x_m]/I_\mcA) \isom \Spec \C[\N \mcA]$.  Such varieties are not always normal.
%\begin{prop} \label{prop:toricNormal}
%For a $d \times m$ integer matrix $\mcA$, the following are equivalent
%\begin{itemize}
%\item[(i)] the affine coordinate ring $\C[\N \mcA] \isom \c[x_1, \dots , x_m]/I_\mcA$ is normal,
%\item[(ii)] the affine toric variety $X_\mcA$ is normal, 
%\item[(iii)] the semigroup $\N \mcA$ is normal,
%\item[(iv)] $\N \mcA = \pos(\mcA) \intersection \Z^d$.
%\end{itemize}
%\end{prop}
%Proposition \ref{prop:toricNormal} also tells us one way to construct the normalization of a  toric variety, namely by normalizing its semigroup \cite{BK}.  Let $\sigma$ be the cone $\{\eta \in \R^d : \eta^{\transpose} \mcA \geq 0 \}$; then $\sigma^{\vee} = \pos (\mcA)$ and the normalization of $X_{\mcA}$ is $X_\sigma$ as defined in \cite{Fulton1993}.
The matrix $\mcA$ defines a map $\mbff_\mcA: \theta \mapsto \theta^\mcA$ from the $d$-dimensional torus $\TT_d$ to the toric variety $X_{\mcA}$.  This gives an explicit torus action and torus embedding.  The closure of the image of $\mbff$ is $X_\mcA$.  %This is the maximal torus orbit corresponding to the interior of the polytope/cone. %more on orbits.
%to big tv summary: Cox homogeneous, Fan as in Fulton, etc.
This is also the parameterization map of an exponential family.

\subsection{Polytopes and projective toric varieties}
Let $\conv(\mcA)$ be the convex hull of the columns of $\mcA$.  This is a polytope.
Let $Y_{\mcA}$ be the projective toric variety defined by taking the closure of the image of $\mbff_\mcA$, and viewing $x_1, \dots, x_m$ as homogeneous coordinates.  The corresponding homogeneous toric ideal is the ideal
\begin{equation} \label{eq:toricideal}
J_\mcA = \langle x^{u^+} - x^{u^-} : u \in \ker \mcA \cap \Z^m, \; \|u^+\|_1 = \|u^-\|_1 \rangle.
\end{equation}
The affine cone over $Y_\mcA$ is the toric variety $X_{\mcA'}$, where $\mcA'$ is $\mcA$ with a row of ones added at the bottom unless the vector of all ones already lies in $\rowspan(\mcA)$.  This induces homogeneity with respect to the $\Z$-grading.  When $\mcA$ has $(1,1, \dots, 1)$ in its row span (e.g. by having equal column sums or $(1, 1, \dots, 1)$ as a row), we say it is $\Z$-graded and the norm restriction in (\ref{eq:toricideal}) is not required.  Instead of $(1,1, \dots, 1)$, we can use another grading of the columns of $\mcA$ to obtain multihomogeneous ideals.

\subsection{The moment map} \label{sec:MomentMap}
The moment map sends a projective toric variety $Y_\mcA$ onto its polytope $\conv(\mcA)$, bijectively on the nonnegative part of the variety.  Theorem \ref{thm:MultiMomentMap} is a version of this result for toric varieties in a product of projective spaces. 

\begin{thm} \label{thm:MomentMap}
Let $\mcA$ be a $d \times m$,  $\Z$-graded matrix, and $Y_\mcA$ the corresponding projective toric variety.  Then the map
\[
\mu: Y_\mcA \ra \conv( \mcA), \;\; \text{given by} \]\[
\mu(y) = \frac{1}{Z(y)} \sum_j |y_j| a_{\cdot j},
\]
where $Z(y) = \sum_j |y_j|$, is a bijection from $Y_{\mcA, \geq 0}$ onto $\conv(\mcA)$.  If further $\rank(\mcA) =d$, with  $\mbff_\mcA$ the torus embedding, then $\mu \circ \mbff_\mcA$ is homeomorphism $\R^d_{> 0} \ra \conv(\mcA)^\circ$. %or extend f_\mcA
\end{thm}

The result is standard and a proof can be found in \cite{GBCP,Fulton1993,Ewald} and goes by the name Birch's theorem in statistics.

\bibliography{mybib}
%\begin{thebibliography}{1}
%\end{thebibliography}
\end{document}